\documentclass[12pt]{amsart}
\usepackage{amsmath,amssymb,amsthm,hyperref}
\usepackage{graphicx}
\usepackage{tikz}

\usepackage{amsmath,amssymb,amsthm,hyperref}
\usepackage{graphicx}
\usepackage{tikz}  
\newtheorem{theorem}{Theorem}
\newtheorem{corollary}[theorem]{Corollary}

\newtheorem{lemma}[theorem]{Lemma}

{\theoremstyle{definition}
	\newtheorem{definition}{Definition}
	\newtheorem{remark}{Remark}

	}

\newcommand{\R}{\ensuremath{\mathbb R}} 

\usepackage{url}

\usepackage{citeall}
\usepackage{amsmath,amssymb,amscd,amsthm,amsxtra, esint, xcolor,mathtools,hyperref}

\title[Malliavin differentiability for the excursion measure of a Gaussian field]{Malliavin differentiability for the excursion measure of a Gaussian field}

\author{Leonardo Maini}
\address{Dipartimento di Matematica e Applicazioni\\
		Universit\`a di Roma Tor Vergata\\
		Via della Ricerca Scientifica, 1, 00133 Roma, Italy}
\email{maini@mat.uniroma2.it}
\thanks{}

\date{\today}
\begin{document}
\maketitle

\medskip

\begin{abstract}
In this work, we investigate the Malliavin differentiability of the excursion volume of a Gaussian field. In particular, we prove that if the moments of the covariance kernel go to $0$ polinomially, then we can determine exactly for which $p\in\mathbb N$ the excursion measure has $p$th square integrable Malliavin derivative. While our approach can not be implemented for Gaussian fields indexed by discrete sets (where the moments do not vanish), when the Gaussian field is indexed by $\R^d$ and stationary, or indexed by $S^d$ and isotropic, the condition on the polynomial decay of the mentioned moments can be inferred starting from simple and general conditions on the covariance function of the field. 
\end{abstract}

\section{Introduction}

Gaussian fields are extremely popular in numerous probabilistic applications for two main reasons: they often arise as very good approximations of general random fields; they are mathematically very convenient. In the last years, they have been studied by many authors in many different settings, see the many references in the bibliography. 

\medskip

In this work, we consider a real-valued measurable Gaussian field $B=(B_x)_{x\in E}$ indexed by a finite measure space $(E,A,\mu)$, that is, a measurable function
\[
B:(E,A)\times (\Omega, \mathcal{F})\,\longrightarrow\,(\R,\mathcal{B}(\R))\,.
\]
Moreover, we assume that $B$ is centered, with unit variance and covariance kernel
\[
K(x,y):={\rm Cov}(B_x,B_y)=\mathbb E[B_x\,B_y]\,\quad\quad K(x,x)=1\,.
\]

In particular, we study the \textbf{excursion measure} of $B$ in $E$, that is 

\[
V(u)=\int_E \mathbf{1}_{B_x\ge u}\,\mu(dx)\,,\quad\quad \mu(E)<\infty.
\]
and ask ourselves if $V(u)$ has $p$-th square integrable Malliavin derivative, namely $V(u)\in\mathbb{D}^{p,2}$, see Section \ref{malliavin general} for more details. 
This type of functional has been extensively studied in the literature, both as a main subject of study and as part of the following more general class of functionals
\begin{equation}
	\label{Y}
	Y(\varphi)=\int_E \varphi(B_x)\,\mu(dx)\,.
\end{equation}
See e.g. \cite{ADP24,BM83,CNN20,GMT24,KN19,LMNP24,Mai23,MN23,MRZ24,MW14,Mar76,NN20,NPP11,NPY,NZ20a,NZ20osc,NZ21,Ros19,Taq77,Tod19} for some references, where the authors worked in three settings: $E$ discrete with $\mu$ counting measure; $E$ Euclidean domain with $\mu$ Lebesgue measure; $E$ the $d$-dimensional sphere with $\mu$ the associated volume measure.

In particular, many of the mentioned works focused on proving central or non central limit theorems for sequences of functionals of the form \eqref{Y}, often trying to provide also quantitative bounds for the rate fo convergence. In this context, studying the Malliavin differentiability of such functionals is crucial, because it enables us to prove quantitative bounds for the convergence in distribution of sequences of functionals of Gaussian fields to a target distribution, which is usually Gaussian. This is made possible by the Malliavin–Stein method, which was developed by Nourdin and Peccati in \cite{NP09}.

Malliavin differentiability  is usually not easy to be determined for general classes of functionals of Gaussian fields and requires ad-hoc analysis depending on the chosen functional, see e.g. \cite{AP20,PS24}, where the authors investigated the Malliavin differentiability of nodal volumes using different appraches. When studying the Malliavin differentiability, one can follow two approaches: trying to exploit the specific structure of the functionals, using for example its geometric meaning or its specific properties, as done in \cite{AP20,PS24}; using the chaotic decomposition of the functional, as done in \cite{MRZ24}. In this paper, we will follow this latter approach. Our goal is characterizing the Malliavin differentiability of the excursion measure of $B$ under assumptions on the covariance structure of $B$. We will work with fields on general measure spaces, and then analyse and simplify the obtained results in the discrete ($E$ finite, $\mu$ counting measure), Euclidean ($E\subseteq \R^d$, $\mu(E)<\infty$, $\mu$ Lebesgue measure) and spherical ($E=S^d$, $\mu$ associated volume measure) cases.

In particular, our approach will be the one used in \cite[Lemma 1.12]{MRZ24}, where it was shown how Malliavin differentiability for the functionals \eqref{Y} can be induced by a control on $\int_{E^2}K(x,y)^q\,dx\,dy$ as $q\rightarrow\infty$.\\

\noindent
{\bf Notations.} From now on, $N(0,1)$ will be used to denote the standard Gaussian distribution and $\phi(x):=e^{-x^2/2}/\sqrt{2\pi}$ will be the associated probability density function. Moreover, given two functions $f(t), g(t)$, 
we will write $f(t) \lesssim g(t)$
if $$\limsup_{t\to\infty}  \frac{f(t)}{g(t)} \in (0,  \infty)\,.$$ We will also write $f(t) \asymp g(t)$ if $f(t) \lesssim g(t)$ and $g(t) \lesssim f(t)$.\\

We are now ready to state Theorem \ref{thm:main}, the main result of this paper.

\begin{theorem}\label{thm:main}
	Let $B=(B_x)_{x\in E}$ be a measurable Gaussian field on a finite measure space $(E,A,\mu)$, with covariance kernel $K$, and assume that $B_x\sim N(0,1)$, $\forall x \in E$. Suppose that
	\begin{equation}
		\label{assumption moments}
		\int_{E^2}K(x,y)^q\,dx\,dy\asymp q^{-\beta}\,,\quad\quad \beta>0\,.
	\end{equation}
	Then, for all $u\in\R$, we have that $V(u)\in\mathbb{D}^{p,2}$ if and only if $p<\beta+1/2$. 
\end{theorem}

\begin{remark}{(Decay of coefficients and smoothness)}
	The result highlights that the faster is the decay of the moments of $K$, the smoother is $V(u)$. As we will see (see \eqref{MalliavinY}), these moments are related to the variances of the chaotic components of $V(u)$, for which we know that a fast decay indicates a smoother functional in the Malliavin sense. This is not a completely new phenomenon in the  literature. Think for example to the analogy to Fourier coefficients, where a faster decay to zero indicates a smoother function. A similar phenomenon has been also observed in other contexts, see e.g. \cite{LS15}, where the authors explored the relation between the angular power spectrum and the regularity  of isotropic Gaussian fields on the sphere.
\end{remark}
\noindent
Let us now comment our result in the discrete, Euclidean and spherical cases:
\begin{itemize}
	\item If $E$ finite, $\mu$ counting measure, $|K(x,y)|<1$ for $x\neq y$, then (Remark \ref{remark})
	\[
	\int_{E^2}K(x,y)^q\,dx\,dy\asymp 1\,,\quad \quad  V(u)\notin \mathbb{D}^{1,2}.
	\]
	\item If $E\subseteq \R^d$, $\mu$ Lebesgue measure, $\mu(E)<\infty$, then Lemma \ref{lemmacruciale} and a "$\alpha$-control" on the dependence/regularity of  $B$ (see \eqref{cond56}-\eqref{minore}) yield
	\[
	\int_{E^2}K(x,y)^q\,dx\,dy\asymp q^{-d/\alpha}\,,\quad \quad  V(u)\in \mathbb{D}^{\lfloor (d/\alpha+1/2)-\rfloor,2}.
	\]
	\item If $E= S^d$, $\mu$ volume measure, then Lemma \ref{lemmacrucialespherical} and a "$\alpha$-control" on the dependence/regularity of $B$ (see \eqref{cond56 spherical}-\eqref{minore spherical}) yield
	\[
	\int_{E^2}K(x,y)^q\,dx\,dy\asymp q^{-d/\alpha}\,,\quad \quad  V(u)\in \mathbb{D}^{\lfloor (d/\alpha+1/2)-\rfloor,2}.
	\]
\end{itemize} 

In particular, in both the Euclidean and spherical case, if $B$ is $C^1$ then $\alpha=2$, implying $\int_{E^2}K(x,y)^q\,dx\,dy\asymp q^{-d/2}$ and $V(u)\in \mathbb{D}^{\lfloor (\frac{d+1}{2})-\rfloor,2}$ . In these two continuous cases, the result can be interpreted as follows: if $\alpha$ is a parameter describing the roughness and local dependence of $B$, a low $\alpha$ (i.e. low local dependence and higher roughness) implies a smoother excursion volume in the Malliavin sense. This phenomenon is similar to the one we can observe when studying local times of the fractional Brownian motion, where a smaller Hurst index (corresponding to a process which is less regular in the Hölder sense) implies a smoother local time of the process, see e.g. \cite{JNP21}. \\

We remark that choosing $\alpha$ small enough, our result could be also combined with Malliavin-Stein method to obtain new quantitative CLTs for excursion measures (or even more general classes of functionals) of Gaussian fields. This research direction is left for future works.\\

The rest of the paper is organized as follows. In Section \ref{malliavin general} we focus on the Malliavin differentiability of a general functional, introducing some notations, preliminaries and showing why our approach is empty and not well suited for the discrete case. In Section \ref{malliavin euclidean spherical} we focus on the Euclidean and spherical cases, explaining how the approach introduced in \cite[Lemma 1.12]{MRZ24}, combined with Theorem \ref{thm:main}, allows to derive the Malliavin differentiability of excursion measures from local regularity and dependency properties of $B$. Finally, Section \ref{section proof} contains the proof of Theorem \ref{thm:main} and some necessary technical results on Mehler'skernel and Mehler's formula.

\section{Preliminaries: Malliavin differentiability for general functionals and a characterization in the discrete case}\label{malliavin general}
The approach we will follow is based on the fact that the Malliavin differentiability of $Y(\varphi)$ can be inferred by the chaotic decomposition of the functional. Since $\varphi\in L^2(\mathbb{R},\phi(x)\,dx)$, we can decompose $\varphi$ with respect to the orthogonal basis of Hermite polynomials $(H_q)_{q\ge0}$
\begin{equation}
	\label{decophi}
	\varphi(x):=\sum_{q=0}^\infty a_q H_q(x)\,,
\end{equation}
where 
\[
a_q(\varphi):=\frac{1}{q!}\int_{\R} H_q(x)\,\varphi(x)\,\phi(x)\,dx\,.
\]
Hermite polynomials can be defined in many different ways. Here we use the following recursive definition
\[
H_0(x)=1\,,\quad H_1(x)=x\,,\quad H_{q+1}(x)=xH_q(x)-qH_{q-1}(x).
\]
Since Hermite polynomials satisfy the Hermite isometry property, namely for $N_1,N_2\sim N(0,1)$ jointly Gaussian
\begin{equation}\label{Herm iso}
	\mathbb{E}[H_q(N_1)H_p(N_2)]=q!\,\mathbf{1}_{p=q}\,(\mathbb{E}[N_1N_2])^q\,,
\end{equation}
we have
\begin{equation}
	\label{normphi}
	\|\varphi\|^2_{L^2(\mathbb R, \phi)}=\sum_{q=0}^\infty q!\,a_q^2<\infty 
\end{equation}
The decomposition \eqref{decophi} implies an othogonal decomposition in $L^2(\Omega)$ for $Y(\varphi)$ (see e.g. \cite{NP12}), the so-called chaotic decomposition, of the form
\[
Y(\varphi)=\sum_{q=0}^\infty Y(\varphi)[q]:=\sum_{q=0}^\infty a_q(\varphi)\,\int_{E}H_q(B_x)\mu(dx)
\]
where $Y(\varphi)[q]$ is said the $q$th chaotic component of $Y(\varphi)$. 

Malliavin differentiability is usually defined in a manner similar to Sobolev spaces, as a differential operator acting on regular functions of elementary Gaussian random variables that generate the Gaussian space on which we are working, and then extending the operator by taking the closure in an $L^p$ space, see for example \cite{NP12,Nua06}. Here we use the following definition, which is more suitable for our purpose. The definitions are equivalent (in the $L^2$ sense), see for example \cite[Section 2.7]{NP12}.

\begin{definition}
	Let us fix $p\in\mathbb N$. We say that $\varphi$ has $p$th (square integrable) Malliavin derivative, i.e. $\varphi\in \mathbb D^{p,2}$, if 
	\[
	\sum_{q=0}q^p\,q!\,a_q^2<\infty.
	\]
	We say that $Y$ has $p$th (square integrable) Malliavin derivative, i.e. $Y\in \mathbb D^{p,2}$, if 
	\begin{equation}
		\label{MalliavinY}
		\sum_{q=0}^\infty q^p {\rm Var}(Y[q])=\sum_{q=0}^\infty q^p\,q!\,a_q^2\,\int_{E^2}K(x,y)^q\,\mu(dx)\mu(dy)<\infty\,.
	\end{equation}
\end{definition}
As a consequence of the previous definition, Malliavin differentiability for functionals of the form \eqref{Y} can be implied by two factors:
\begin{itemize}
	\item The Malliavin regularity of $\varphi$, i.e., the decay as $q\rightarrow\infty$ of the coefficients
	
	\begin{equation}
		\label{aq}
		q\mapsto q!\,a_q^2(\varphi):=q!\left(\int_{\R} H_q(x)\,\varphi(x)\,\phi(x)\,dx\right)^2
	\end{equation}
	\item A low correlation of $B_x$ and $B_y$ for $x,y\in E$, implying a good decay as $q\rightarrow\infty$ of the double integral
	\[
	q\mapsto \int_{E^2} \rho^q(x,y)\,\mu(dx)\,\mu(dy)
	\]
\end{itemize}
What we can notice, for example, is that $|K|\le 1$ immediately implies that $\varphi\in \mathbb{D}^{p,2}\implies Y(\varphi)\in\mathbb{D}^{p,2}.$ Nevertheless, the viceversa is not always true. To highlight this fact we state the following criterion, which can be easily derived as a combination of \eqref{normphi} and \eqref{MalliavinY}.

\begin{theorem}\label{thm:main1}
	Assume $B$ Gaussian field with $B_x\sim N(0,1)$ for every $x$. Then:
	\begin{itemize}
		\item If $\int_{E^2}K(x,y)^q\mu(dx)\mu(dy)\lesssim q^{-\beta}$, $\beta>0$, with $p\le \beta$, then $Y(\varphi)\in\mathbb D^{p,2}$.
		\vspace{2mm}
		\item If $\int_{E^2}K(x,y)^q\mu(dx)\mu(dy)\asymp 1$, then $Y(\varphi)\in\mathbb D^{p,2}$ $\iff$ $\varphi\in \mathbb{D}^{p,2}$.
	\end{itemize} 
\end{theorem}

\begin{remark}{(Malliavin differentiability for discrete functionals).}\label{remark}
	When $E$ finite and $\mu$ is the counting measure, we have
	\[
	\int_{E^2}K(x,y)^q\mu(dx)\mu(dy)=|E|+\int_{E^2\,,\,x\neq y}K(x,y)^q\mu(dx)\mu(dy)\,.
	\]
	Thus, assuming $K(x,y)<1$ for $x\neq y$, we have $Y(\varphi)\in\mathbb D^{p,2}$ $\iff$ $\varphi\in \mathbb{D}^{p,2}$, meaning that in the discrete case we can not improve the Malliavin differentiability of the functional giving conditions on $B$. As we are going to see, things are very different in the continuous settings of Euclidean and spherical functionals, where in general $\int_{E^2}K(x,y)^q\mu(dx)\mu(dy)$ can vanish polynomially as $q\rightarrow\infty$.
\end{remark}

\section{Malliavin differentiability of Euclidean/spherical functionals and excursion volumes}\label{malliavin euclidean spherical}
When $E\subseteq \mathbb{R}^d$ compact, $\mu$ Lebesgue measure, $0<\mu(E)<\infty$, things are very different from the discrete case analyzed in Remark \ref{remark}. Indeed, if we have some uniform control on the dependence in $B$ (see conditions \eqref{cond56}-\eqref{cond56 spherical}), we can derive bounds of the form 
\[
\int_{E^2}K(x,y)^q\mu(dx)\mu(dy)\asymp \int_{E^2\,,\,{\rm dist}(x,y)<\epsilon}K(x,y)^q\mu(dx)\mu(dy)\lesssim q^{-\beta}
\]
with $\beta>0$, that "helps" $Y(\varphi)$ to be Malliavin differentiable even when $\varphi$ is very irregular in the Malliavin sense, recall \eqref{MalliavinY}. For Eculidean and spherical functionals, where that measure $\mu$ is the suitable volume measure on the space we consider, $V(u)$ is said the \textbf{exucrsion volume} of $B$ in $E$ at level $u\in\R$.

\subsection{Euclidean functionals}
Malliavin differentiability for local functionals was studied in \cite{AP20,PS24}, where the authors focused on the regularity of nodal volumes. As far as we know, for our specific class of functionals \eqref{Y} the first result on the Malliavin differentiability of $Y(\varphi)$ was proved in \cite{MRZ24}, as a consequence of the following lemma.
\begin{lemma}{\cite[Lemma 1.12]{MRZ24}}\label{lemmacruciale} Suppose $B$ Gaussian field on $\R^d$ stationary with continuous covariance function  $C(x-y):=K(x,y)$, $C(0)=1$. Let $C_1, C_2, \alpha, \delta, \epsilon$ be positive constants. 
	Suppose that the covariance function $C$
	satisfies the following bounds:
	\begin{align} \label{cond56}
		| C(x) | \leq C_1 \|x\|^{-\delta}, \,\, \forall x\in\R^d
		\quad
		\text{and}
		\quad
		1 - C(y) \ge C_2  \|x\|^\alpha
		\,\,\text{for $\|x\| \leq \epsilon$.}
	\end{align}
	Then,  there exists $c>0$ such that 
	$$
	\int_{\R^d}|C(z)|^N dz 
	\le\,c\, N^{-\frac{d}{\alpha}}
	$$
	for any integer $N \geq \frac{d}\delta +1$.  Moreover if we additionally assume that 
	\begin{equation}\label{minore}
		1-C(y) \le C_3 \|x\|^{\alpha}
	\end{equation}
	for some $C_3>0$ and for $\|x\|\le \epsilon$, we have 
	\begin{equation}\label{stima_momenti_gen}
		\int_{\R^d}|C(z)|^N dz 
		\asymp N^{-\frac{d}{\alpha}}.
	\end{equation} 
	
\end{lemma} 
\begin{remark}\label{remark 3}
	The previous results holds even if we take $C^N$ instead of $|C|^N$. Indeed, under the same assumptions of Lemma \ref{lemmacruciale} one can prove 
	\[
	\int_{\R^d}|C^N(z)|\,dz\asymp\int_{\|z\|<\epsilon}|C^N(z)|\,dz=\int_{\|z\|<\epsilon}C^N(z)\,dz\asymp N^{-d/\alpha}\,,
	\]
	for every $\epsilon>0$ arbitrary small. See \cite[Remark 4.1]{MRZ24}
\end{remark}
Condition \eqref{cond56} is not a Hölder condition and can be seen as a condition giving control on the local dependence of the Gaussian field. On the other hand, Condition \eqref{minore} is a Hölder regularity condition.

If $B$ is $C^1$ and stationary, then both the conditions are satisfied with $\alpha=2$. This is the case for example for the famous Berry random wave model, a smooth Gaussian field which has been extensively studied in the last years, see e.g. \cite{MN23,NPR19,PV20,Vid21}. If $B$ is not $C^1$, then the conditions could be satisfied with $\alpha\in(0,2)$. For example if $C$ has the following form
\[
C(x)=e^{-\|x\|^{\alpha}}\quad\quad \alpha\in(0,2]\,
\]
then we have 
\[
C(x)=1-\alpha \|x\|^{\alpha}+o(\|x\|^{\alpha})\,,\quad \quad \|x\|\rightarrow\,0\,.
\]
As an easy consequence of Lemma \ref{lemmacruciale} and Theorem \ref{thm:main1}, we obtain the following result, proved for the first time in \cite{MRZ24}.

\begin{corollary}
	Suppose $B$ Gaussian field on $\R^d$ centered and stationary with continuous covariance function  $C(x-y):=K(x,y)$, $C(0)=1$. If condition \eqref{cond56} holds and $p\le d/\alpha$, then $Y(\varphi)\in \mathbb D^{p,2}$.
\end{corollary}
\begin{proof}
	By Theorem \ref{thm:main1}, we only need to prove that 
	\[
	\int_{E^2}C(x-y)^q\,dx\,dy\lesssim q^{-d/\alpha}\,.
	\]
	This follows immediately by condition \eqref{cond56} and Lemma \ref{lemmacruciale}
	\begin{equation}
		\label{upper bound}
		\int_{E^2}C(x-y)^q\,dx\,dy=\int_{\R^d}C(z)^q\,{\rm Vol}(E\cap (E+z))\,dz\le {\rm Vol}(E)q^{-d/\alpha}\,.
	\end{equation}
\end{proof}
Note that even if $\varphi$ is very irregular in the Malliavin sense, if I choose $\alpha$ small enough (i.e. low enough local dependence) I have that $Y(\varphi)$ is Malliavin differentiable. 

In particular, for the specific case of excursion volumes, combining Theorem \ref{thm:main} and Lemma \ref{lemmacruciale}, we can characterize the Malliavin differentiability of $V(u)$ under conditions \eqref{cond56}-\eqref{minore}.

\begin{theorem}
	Suppose $B$ Gaussian field on $\R^d$ centered and stationary with continuous covariance function  $C$, $C(0)=1$. If conditions \eqref{cond56}-\eqref{minore} hold, then $V(u)\in\mathbb{D}^{p,2}$ $\iff$ $p<d/\alpha+1/2$.
\end{theorem}
\begin{proof}
	By Theorem \ref{thm:main}, we only need to prove that 
	\[
	\int_{E^2}C(x-y)^q\,dx\,dy\asymp q^{-d/\alpha}\,.
	\]
	First of all, we can write 
	\[
	\int_{E^2}C(x-y)^q\,dx\,dy=\int_{\|z\|\le \epsilon}|C(z)|^q\,{\rm Vol}(E\cap (E+z))\,dz+O\left(\int_{\|z\|> \epsilon}|C(z)|^q\,\,dz\right)\,.
	\]
	Here $\epsilon>0$ was chosen small enough to have $C(z)$ positive for $|z|\le \epsilon$ and to use conditions \eqref{cond56}-\eqref{minore}. Moreover, the function $z\mapsto {\rm Vol}(E\cap (E+z))$ is continuous in $0$ (see \cite{G11}), so we can choose $\epsilon$ small enough to have ${\rm Vol}(E\cap (E+z))>{\rm Vol}(E)/2>0$ for $\|z\|\le \epsilon$. With this choice, by conditions \eqref{cond56}-\eqref{minore} and Lemma \ref{lemmacruciale} we have (see also Remark \ref{remark 3})
	\[
	\int_{\|z\|\le \epsilon}|C(z)|^q\,{\rm Vol}(E\cap (E+z))\,dz\asymp \int_{\|z\|\le \epsilon}|C(z)|^q\,dz\asymp \int_{\R^d}|C(z)|^q\,dz\asymp q^{-d/\alpha}\,.
	\]
	Moreover, by reasoning as in \cite[Remark 4.1]{MRZ24}, one can prove
	\[
	\int_{\|z\|> \epsilon}|C(z)|^q\,\,dz=o(q^{-d/\alpha})\,.
	\]
\end{proof}

\subsection{Spherical functionals}
In this subsection, we study the case where $E=S^d$ and $\mu$ is the associated volume measure.
Following an analogous approach, we can study local conditions of the form \eqref{cond56}-\eqref{minore} in the spherical case (see conditions \eqref{cond56 spherical}-\eqref{minore spherical}), to derive Malliavin differentiability for analogous functionals. In this case these conditions are also related to the decay of the angular power spectrum $(C_\ell)_{\ell\in\mathbb N}$ of the field, see e.g. \cite{DiL25,DMSV25,LS15} for more details and precise statements. 

First of all, let us start with some basic notions. In this subsection, we will assume that the Gaussian field $B$ is isotropic, i.e. its distribution is invariant under the action of the group of rotations $SO(d)$. Under this assumption, the covariance kernel $K$ has the following representation
\[
K(x,y)=\kappa(\langle x,y \rangle)=\sum_{\ell=0}^\infty C\ell\, \frac{n_{\ell,d}}{\omega_{d}}\,G_{\ell,d}(\langle x,y \rangle)\,,\quad\quad x,y\in S^d\,,
\]
where $\kappa:[-1,1]\rightarrow[-1,1]$ is the covariance function of $B$ and
\begin{itemize}
	\item $(C_\ell)_{\ell\in\mathbb N}$, $C_{\ell}\ge0$, is the angular power spectrum of $B$;
	\item $n_{\ell,d}$ is the number of linearly indipendent homogeneous harmonic polynomials of degree $\ell$ in $d+1$ variables
	\[
	n_{\ell,d}=\frac{2\ell+d-1}{\ell}\binom{\ell+d-2}{\ell-1}\sim \ell^{d-1}\,,\quad\quad \text{ as }\ell\rightarrow\infty\,;
	\]
	\item $\omega_d=\mu(S^d)=\frac{2\pi^{\frac{d+1}{2}}}{\Gamma(\frac{d+1}{2})}$ is the total spherical measure;
	\item $(G_{\ell,d})_{\ell\in\mathbb N}$ are the normalized Gegenbauer polynomials (see e.g. \cite[4.7]{Sze39}).
\end{itemize}
\noindent
Let us now prove an analogous of Lemma \ref{lemmacruciale} in the spherical case. 

\begin{lemma}\label{lemmacrucialespherical}
	Let $B$ be a centered, isotropic Gaussian field on $S^d$ with covariance function $\kappa$, $\kappa(1)=1$, and angular power spectrum $(C_\ell)_{\ell\in\mathbb N}$. Let $C_1,C_2,\alpha,\epsilon$ be positive constants. 	Suppose $\kappa:[-1,1]\rightarrow\R$ is continuous and
	satisfies
	\begin{align} \label{cond56 spherical}
		|\kappa(u)| <1 \text{ for } u\neq 1, \,\,
		\quad
		\text{and}
		\quad
		1 - \kappa(\cos(\theta)) \ge C_1  |\theta|^\alpha
		\,\,\text{for $|\theta| \leq \epsilon$.}
	\end{align}
	Then, there exists $c>0$ such that 
	$$
	\int_{(S^d)^2}|\kappa(\langle x,y\rangle)|^N \mu(dx)\mu(dy)
	\le\,c\, N^{-\frac{d}{\alpha}}
	$$
	Moreover, if we additionally assume that 
	\begin{equation}\label{minore spherical}
		1-\kappa(cos(\theta)) \le C_2 |\theta|^{\alpha}
	\end{equation}
	for some $C_2>0$ and for $|\theta|\le \epsilon$, we have 
	\begin{equation}\label{stima_momenti_gen spherical}
		\int_{(S^d)^2}|\kappa(\langle x,y\rangle)|^N \mu(dx)\mu(dy)
		\asymp N^{-\frac{d}{\alpha}}.
	\end{equation} 
\end{lemma}
\begin{proof}
	First of all, we can write
	\begin{equation}\label{inthspherical}
		\int_{E^2}|\kappa(\langle x,y\rangle)|^N\,\mu(dx)\,\mu(dy)=\omega_d\omega_{d-1}\int_0^\pi |\kappa(cos(\theta))|^N\sin(\theta)^{d-1}d\theta \,.
	\end{equation}
	Moreover, we can split the integral \eqref{inthspherical} in two parts
	\[
	\eqref{inthspherical}= \omega_d\omega_{d-1}\int_0^\epsilon |\kappa(cos(\theta))|^N\sin(\theta)^{d-1}d\theta + O\left(\max_{|u|>\epsilon} |\kappa(u)|^N\right)\,.
	\]
	Finally, using condition \eqref{cond56 spherical} we conclude observing that $\max_{|u|>\epsilon} |\kappa(u)|<1$ and 
	\[
	\int_0^\epsilon |\kappa(cos(\theta))|^N\sin(\theta)^{d-1}d\theta\lesssim \int_0^\epsilon (1-C_2\,\theta^{\alpha})^N\,\theta^{d-1}d\theta\asymp N^{-d/\alpha}\,.
	\]
	where the last asymptotics follows reasoning as in \cite[Equation (4.3)]{MRZ24}. The proof of \eqref{stima_momenti_gen spherical} is analogous and follows by  condition \eqref{minore spherical} with the inverse inequality.
\end{proof}

Exactly as in the Euclidean case, here we can derive the moments'asymptotics of $K$ under conditions on the low dependence (\eqref{cond56 spherical}) and low Hölder regularity/ high roughness (\eqref{minore}) of our field $B$.

Even in this case, if $B$ is $C^1$ both the conditions hold with $\alpha=2$, while in general one could have a field with only Hölder continuous realizations, satisfying conditions \eqref{cond56 spherical}-\eqref{minore spherical} with arbitrary small $\alpha$. We remark that these conditions are also related to the decay of the angular power spectrum of the field, see e.g \cite{DiL25,DMSV25,LS15} for more details. Thus, in this case one could build many examples with specific choices for the sequence $C_\ell$, obtaining random fields which are not $C^1$, but satisfy the conditions of Lemma \ref{lemmacrucialespherical} for abitrary small $\alpha$. 

Combining Lemma \ref{lemmacrucialespherical} with Theorem \ref{thm:main1} and Theorem \ref{thm:main}, we obtain analogous results to the the Euclidean case.

\begin{corollary}
	Suppose $B$ centered, isotropic Gaussian field on $S^d$ with continuous covariance function  $\kappa(\langle x,y \rangle):=K(x,y)$, $\kappa(1)=1$. If condition \eqref{cond56 spherical} holds and $p\le d/\alpha$, then $Y(\varphi)\in \mathbb D^{p,2}$.
\end{corollary}
\begin{theorem}
	Suppose $B$ centered, isotropic Gaussian field on $S^d$ with continuous  covariance function  $\kappa$, $\kappa(1)=1$. If conditions \eqref{cond56 spherical}-\eqref{minore spherical} hold, then\\ $V(u)\in\mathbb{D}^{p,2}$ $\iff$ $p<d/\alpha+1/2$.
\end{theorem}

\section{Mehler's kernel/formula and proof of Theorem \ref{thm:main}}\label{section proof}
This last section is devoted to the proof of Theorem \ref{thm:main}, our main result. We start introducing Mehler's kernel and Mehler's formula. \textbf{Mehler's kernel} is defined for $u,v\in\R$ and $r\in(-1,1)$ as follows
\begin{equation}
	\label{Mehlerdef}
	M_r(u,v):= \frac{1}{\sqrt{1-r^2}} \exp\left( \frac{2uvr - (u^2+v^2)r^2}{2(1-r^2)} \right)\,.
\end{equation}
It can be also defined as the ratio of the probability density function of $(N,rN+\sqrt{1-r^2}N')$ and that of $(N,N')$, where $N,N'$ are independent standard Gaussian random variables, namely
\begin{equation}
	\label{Mehler as ratio}
	M_r(u,v):= \frac{\frac{1}{2\pi\sqrt{1-r^2}}\exp\left( -\frac{1}{2(1-r^2)}(u^2+v^2-2uvr) \right)}{\frac{1}{2\pi}\exp\left( -\frac{1}{2}(u^2+v^2) \right)}
\end{equation}
A very important and well known fact in harmonic analysis is that Mehler's kernel can be expanded in terms of Hermite polynomials. The explicit formula we obtain (see \eqref{eq: Mehlerformula}) is called \textbf{Mehler's formula}.\\

There exist several proofs of the formula in the literature. Here we provide a simple probabilistic proof for the sake of completeness.
\begin{lemma}{(Mehler's formula)}
	For every $u,v\in \R$ and $r\in(-1,1)$, we have
	\begin{equation}
		\label{eq: Mehlerformula}
		M_r(u,v)=\sum_{q=0}^\infty H_q(u)H_q(v)\frac{r^q}{q!}\,.
	\end{equation}
\end{lemma}
\begin{proof}
	Fix $r\in(-1,1)$ and denote by $(N,N')$ two indipendent standard Gaussian random variables. First of all, note that we can write two different expressions for the covariance between $N$ and $rN+ \sqrt{1-r^2}N'$, using the bivariate density (expressed in terms of Mehler's kernel, see \eqref{Mehler as ratio}) 
	\begin{align*}
		{\rm Cov}\left(\mathbf{1}_N\ge u\,,\,\mathbf{1}_{rN+\sqrt{1-r^2}N'}\ge v\right) &=\int_u^\infty \int_v^\infty (M_r(x,y)-1)\phi(x)\phi(y)\,dx\,dy\,.
	\end{align*}
	or using the orthogonal expansion \eqref{chao deco indicator} and Hermite isometry \eqref{Herm iso}
	\begin{align*}
		{\rm Cov}\left(\mathbf{1}_N\ge u\,,\,\mathbf{1}_{rN+\sqrt{1-r^2}N'}\ge v\right) &=\sum_{q=1}^\infty H_{q-1}(u)H_{q-1}(v)\phi(u)\phi(v)\frac{r^{q}}{q!}\\
		&=\int_u^\infty\int_v^\infty\left(\sum_{q=1}^\infty H_{q}(x)H_{q}(y)\phi(x)\phi(y)\frac{r^{q}}{q!}\right)\,dx\,dy\,.
	\end{align*}
	where the last equality follows by using $H_{q-1}(u)\phi(u)=\int_u^\infty H_q(x)\phi(x)\,dx$ and Fubini theorem (note that when $|r|<1$ the series $\sum_{q=0}^\infty H_q(x)H_q(y)r^q/q!$ is absolutely convergent, since $|H_q(x)|\le \sqrt{q!}e^{x\sqrt q}$, see e.g. \cite[(1.2)]{EM90}). Thus, differentiating both expressions, we obtain
	\[
	(M_r(u,v)-1)\phi(u)\phi(v)=\sum_{q=1}^\infty H_{q}(u)H_{q}(v)\phi(u)\phi(v)\frac{r^{q}}{q!}
	\]
	for every $u,v\in\R$, which implies the equality \eqref{eq: Mehlerformula}. 
\end{proof}
Now that Mehler's kernel and Mehler's formula have been introduced, we are able to prove the following lemma, that will be a fundamental tool in the proof of Theorem \ref{thm:main}. 

\begin{lemma}
	Let $(H_q)_{q\ge0}$ be the Hermite polynomials defined in Section \ref{malliavin general}. Then  
	\begin{equation}
		\label{series beta}
		\sum_{q=0}^\infty H_q^2(u)\frac{1}{q!}\frac{1}{q^{\gamma}}<\infty \quad \iff\quad  \gamma>1/2\,\,.
	\end{equation}
\end{lemma}
\begin{proof}
	We start observing that by Stirling's approximation we have, as $q\rightarrow\infty$,
	\[
	\frac{1}{q^{\gamma}}\asymp\frac{1}{(q+1)^{\gamma}}\asymp{\rm Beta}(\gamma,q+1)\asymp\int_0^1 r^{q}(1-r)^{\gamma-1}\,dr\,.
	\]
	Using this fact, Fubini-Tonelli's theorem and Mehler's formula, we have that the series in \eqref{series beta} is convergent if and only if
	\[
	\int_0^1\left(\sum_{q=0}^\infty H_q^2(u)\frac{r^q}{q!}\right)\,(1-r)^{\gamma-1}dr=\int_0^1\,M_r(u,u)\,(1-r)^{\gamma-1}dr<\infty\,.
	\]
	Thus, recalling \eqref{Mehlerdef}, we conclude observing that
	\[
	\int_0^1\exp\left( \frac{2u^2r}{2(1+r)} \right) (1-r)^{\gamma-3/2}\,dr<\infty\quad \iff \quad \gamma >1/2\,.
	\]
\end{proof}

\subsection{Proof of Theorem \ref{thm:main}}

We are finally ready to prove Theorem \ref{thm:main}. First of all, we have the following orthogonal decomposition for $\varphi=\mathbf{1}_{[u,\infty)}$ in $L^2(\mathbb R, \phi(x)\,dx)$
\begin{equation}
	\label{chao deco indicator}
	\varphi=\mathbf{1}_{[u,\infty)}=\sum_{q=0}^\infty a_q(\varphi)\,H_q\,,
\end{equation}
where $$a_q(\varphi)=\frac{1}{q!}\int_\R \varphi(x)\,H_q(x)\,\phi(x)\,dx=\frac{1}{q!}\int_u^\infty H_q(x)\,\phi(x)\,dx=\frac{1}{q!}\,H_{q-1}(u)\phi(u) \,,$$
where the last equality follows by integration by parts, the recursive definition of Hermite polynomials and the fact that $H_q'=qH_{q-1}$. As a consequence, by \eqref{MalliavinY}, $V(u)\in\mathbb{D}^{p,2}$ if and only if
\[
\sum_{q=1}^\infty  \frac{q^p}{q!}H_{q-1}^2(u)\int_{E^2}K(x,y)^q\,\mu(dx)\,\mu(dy)<\infty\,.
\]
Thus, since we are assuming \eqref{assumption moments} and the series is with positive terms, we have
\[
V(u)\in\mathbb{D}^{p,2}\quad \iff \quad \sum_{q=1}^\infty  \frac{q^{p-\beta}}{q!}H_{q-1}^2(u)<\infty\quad \iff \quad \sum_{q=0}^\infty  \frac{q^{p-\beta-1}}{q!}H_{q}^2(u)<\infty\,
\]
Thus, the proof is concluded by applying \eqref{series beta} with $\gamma=\beta+1-p$.

\section*{Acknowledgments}
We acknowledge financial support from MUR Excellence Department Project MatMod@TOV awarded to the Department of Mathematics, University of Rome Tor Vergata, CUP E83C18000100006. We acknowledge support from the MUR 2022 PRIN project GRAFIA, project code 202284Z9E4. We acknowledge support from INdAM group GNAMPA.

%
%


\begin{thebibliography}{6}
	%
	\bibitem{ADP24}
	
	J.~Angst,  F.~Dalmao, G.~Poly,
	{\it A total variation version of Breuer-Major central limit theorem under 
		$\mathbb{D}^{1,2}$ assumption,}
	Electron. Commun. Probab.   29 (2024), Paper No. 13, 8 pp.
	
	\bibitem{AP20}  
	
	J.~Angst,  G.~Poly,
	{\it  On the absolute continuity of random nodal volumes,}
	Ann. Probab.   48 (2020), no. 5, 2145--2175.
	
	
	
	
	\bibitem{BCP}
	
	V.~Bally,  L.~Caramellino,G.~Poly, 
	{\it Regularization lemmas and convergence in total variation},
	Electron. J. Probab.   25 (2020), Paper No. 74, 20 pp.
	
	\bibitem{BCW20} 
	D.~Beliaev. V.~Cammarota. I.~Wigman,
	{\it No repulsion between critical points for planar Gaussian random fields,}
	Electron. Commun. Probab. (2020) 25 1--13.
	
	\bibitem{BH23}
	D.~B. Beliaev and A. Hegde, On convergence of volume of level sets of stationary smooth Gaussian fields, Electron. Commun. Probab. {\bf 28} (2023), Paper No. 42, 9 pp.; MR4684055
	
	\bibitem{BMM24} 
	D.~B. Beliaev, M. McAuley and S. Muirhead, A central limit theorem for the number of excursion set components of Gaussian fields, Ann. Probab. {\bf 52} (2024), no.~3, 882--922; MR4736695
	
	\bibitem{BMM25} 
	D.~B. Beliaev, M. McAuley and S. Muirhead, A covariance formula for the number of excursion set components of Gaussian fields and applications, Ann. Inst. Henri Poincar\'e{} Probab. Stat. {\bf 61} (2025), no.~1, 713--745
	
	
	\bibitem{BDMT24} 
	S. Bourguin et al., Spherical Poisson waves, Electron. J. Probab. {\bf 29} (2024), Paper No. 8, 27 pp.; MR4688685
	
	\bibitem{BM83} 
	P.~Breuer, P.~Major,
	{\it Central limit theorems for nonlinear functionals of Gaussian fields,}
	J. Multivariate Anal.   13 (1983), no. 3, 425--441.
	
	
	\bibitem{CM18}
	V. Cammarota and D. Marinucci, A quantitative central limit theorem for the Euler-Poincar\'e{} characteristic of random spherical eigenfunctions, Ann. Probab. {\bf 46} (2018), no.~6, 3188--3228; MR3857854
	
	
	\bibitem{CNN20}
	
	S.~Campese, I.~Nourdin, D.~Nualart,
	{\it Continuous Breuer-Major theorem: tightness and nonstationarity},
	Ann. Probab. 48 (2020), no. 1, 147--177.
	
	
	
	
	\bibitem{CGR24} 
	L.~Caramellino, G.~Giorgio, M.~Rossi,
	{\it Convergence in total variation for nonlinear functionals of random hyperspherical harmonics},
	J. Funct. Anal. 286 (2024), no. 3, Paper No. 110239, 32 pp.
	
	\bibitem{CDV21}
	A. Caponera, C. Durastanti and A. Vidotto, LASSO estimation for spherical autoregressive processes, Stochastic Process. Appl. {\bf 137} (2021), 167--199; MR4244190
	
	\bibitem{CM21}
	A. Caponera and D. Marinucci, Asymptotics for spherical functional autoregressions, Ann. Statist. {\bf 49} (2021), no.~1, 346--369; MR4206681
	
	
	
	\bibitem{CX24}
	D. Cheng and Y.~M. Xiao, The expected Euler characteristic approximation to excursion probabilities of Gaussian vector fields, Ann. Appl. Probab. {\bf 34} (2024), no.~6, 5664--5693; MR4839635
	
	\bibitem{DiL25} S. Di Lillo, Critical Points of Random Neural Networks, preprint. arXiv preprint arXiv:2505.17000
	
	
	\bibitem{DMSV25} S. Di Lillo, D. Marinucci, M. Salvi, S. Vigogna, Fractal and Regular Geometry of Deep Neural Networks, preprint. arXiv preprint arXiv:2504.06250
	
	\bibitem{DM79} 
	R.L.~Dobrushin, P.~Major,
	{\it Non-central limit theorems for nonlinear
		functionals of Gaussian fields}. Z. Wahrsch. Verw. Gebiete 
	(1979) 50, no. 1, 27--52.
	
	
	\bibitem{DRRV23} H. Duminil-Copin et al., Existence of an unbounded nodal hypersurface for smooth Gaussian fields in dimension $d\geq 3$, Ann. Probab. {\bf 51} (2023), no.~1, 228--276
	
	
	\bibitem{DMT24} C. Durastanti, D. Marinucci and A.~P. Todino, Flexible-bandwidth needlets, Bernoulli {\bf 30} (2024), no.~1, 22--45; MR4665568
	
	\bibitem{EM90}
	S.~J.~L. van~Eijndhoven and J.~L.~H. Meyers, New orthogonality relations for the Hermite polynomials and related Hilbert spaces, J. Math. Anal. Appl. {\bf 146} (1990), no.~1, 89--98; MR1041203
	
	\bibitem{FHMNP25}
	S. Favaro et al., Quantitative CLTs in deep neural networks, Probab. Theory Related Fields {\bf 191} (2025), no.~3-4, 933--977
	
	\bibitem{G11}
	B. Galerne, Computation of the perimeter of measurable sets via their covariogram. Applications to random sets, Image Anal. Stereol. {\bf 30} (2011), no.~1, 39--51; MR2816305
	
	\bibitem{Gas23}
	L. Gass, Cumulants asymptotics for the zeros counting measure of real Gaussian processes, Electron. J. Probab. {\bf 28} (2023), Paper No. 151, 45 pp.; MR4669753
	
	\bibitem{GS24}
	L. Gass and M. Stecconi, The number of critical points of a Gaussian field: finiteness of moments, Probab. Theory Related Fields {\bf 190} (2024), no.~3-4, 1167--1197; MR4811822
	
	\bibitem{GMT24} 
	F.~Grotto, L.~Maini and A.P.~Todino,
	{\it Fluctuations of polyspectra in spherical and Euclidean random wave models,}
	Electron. Commun. Probab. (2024), Paper No.9,  12 pp.
	
	\bibitem{JNP21}
	Arturo Jaramillo. Ivan Nourdin. Giovanni Peccati. "Approximation of fractional local times: Zero energy and derivatives." Ann. Appl. Probab. 31 (5) 2143 - 2191, October 2021. https://doi.org/10.1214/20-AAP1643
	
	\bibitem{KN19} 
	S.~Kuzgun, D.~Nualart,
	{\it Rate of convergence in the Breuer-Major theorem via chaos
		expansions,}
	Stoch. Anal. Appl.  (2019) 37, no. 6, pp. 1057--1091. 
	
	
	\bibitem{LS15}
	A. Lang and C. Schwab, Isotropic Gaussian random fields on the sphere: regularity, fast simulation and stochastic partial differential equations, Ann. Appl. Probab. {\bf 25} (2015), no.~6, 3047--3094.
	\bibitem{LMNP24}
	
	N.~Leonenko, L.~Maini, I.~Nourdin, F.~Pistolato,
	{\it Limit theorems for p-domain functionals of stationary Gaussian fields},
	Electron. J. Probab. 29 (2024),  Paper No. 136, 33 pp
	
	\bibitem{LMO22}
	N.~N. Leonenko, A.~A. Malyarenko and A.~Y. Olenko, On spectral theory of random fields in the ball, Theory Probab. Math. Statist. No. 107 (2022), 61--76; MR4511144
	
	\bibitem{LRM25}
	N.~N. Leonenko and M.-D. Ruiz-Medina, High-level moving excursions for spatiotemporal Gaussian random fields with long range dependence, J. Stat. Phys. {\bf 192} (2025), no.~2, Paper No. 19, 29 pp.; MR4855058
	
	
	
	\bibitem{Mai23}
	L.~Maini, 
	{\it Asymptotic covariances for functionals of weakly stationary random fields}, 
	Stochastic Process. Appl., 170, (2024), Paper no. 104297.
	
	
	\bibitem{MN23}
	L.~Maini, I.~Nourdin,
	{\it Spectral central limit theorem for additive functionals of isotropic 
		and stationary Gaussian fields}, 
	Ann. Probab., 52, (2024), no. 2, 737-763.
	
	\bibitem{MRZ24}
	L.~Maini, M. Rossi, G. Zheng,
	{\it Almost sure central limit theorems via chaos expansions and related results}, 
	preprint (2025+). https://arxiv.org/abs/2502.00759
	
	\bibitem{MW14}
	D.~Marinucci, I.~Wigman,
	{\it On nonlinear functionals of random spherical eigenfunctions},
	Comm. Math. Phys.,
	327, (2014), no. 3, 849--872.
	
	
	\bibitem{Mar76}
	G.~Maruyama, 
	{\it  Nonlinear functionals of gaussian stationary processes 
		and their applications,} 
	In: Maruyama G., Prokhorov J.V. 
	(eds) Proceedings of the Third Japan-USSR Symposium
	on Probability Theory. Lecture Notes in Mathematics (1976) vol 550. 
	Springer, Berlin, Heidelberg.
	
	\bibitem{MRVK23}
	S. Muirhead et al., The phase transition for planar Gaussian percolation models without FKG, Ann. Probab. {\bf 51} (2023), no.~5, 1785--1829; MR4642224
	
	\bibitem{NN20}
	I.~Nourdin, D.~Nualart,
	{\it  The functional Breuer-Major theorem},
	Probab. Theory Related Fields 176 (2020), no. 1-2, 203--218.
	
	
	\bibitem{NNP21}
	
	I.~Nourdin,  D.~Nualart,  G.~Peccati, 
	{\it The Breuer-Major theorem in total variation: 
		improved rates under minimal regularity,}
	Stochastic Process. Appl. 131 (2021), 1--20.
	
	
	
	
	
	\bibitem{NP09}
	
	I.~Nourdin,  G.~Peccati, 
	{\it Stein's method on Wiener chaos,}
	Probab. Theory Related Fields   145 (2009), no. 1-2, 75--118.
	
	
	
	\bibitem{NP12} I.~Nourdin and G.~Peccati, 
	{\it Normal approximations with Malliavin calculus: 
		from Stein's method to universality}, 
	Cambridge Tracts in Mathematics 192. 
	Cambridge University Press,
	Cambridge, 2012, xiv+239.
	
	\bibitem{NPP11} 
	I.~Nourdin, G.~Peccati, M.~Podolskij, 
	{\it Quantitative Breuer-Major theorems},
	Stochastic Process. Appl., 121 (2011), no.4,  793-812.
	
	
	\bibitem{NPR19}
	I.~Nourdin, G.~Peccati, M.~Rossi,  
	{\it Nodal statistics of planar random waves,}
	Comm. Math. Phys. 369 (2019), no.1, 99--151.
	
	\bibitem{NPY}
	I.~Nourdin, G.~Peccati, X.~Yang, 
	{\it Berry-Esseen bounds in the Breuer-Major CLT 
		and Gebelein's inequality,}
	Electron. Commun. Probab.   24 (2019), 
	Paper No. 34, 12 pp.
	
	
	
	\bibitem{Nua06}
	
	D.~Nualart, 
	{\it The Malliavin calculus and related topics,}
	Second edition
	Probab. Appl. (N. Y.)
	Springer-Verlag, Berlin, 2006. xiv+382 pp.
	
	\bibitem{NOL08}
	D.~Nualart,  S.~Ortiz-Latorre,
	{\it  Central limit theorems for multiple stochastic integrals 
		and Malliavin calculus,}
	Stochastic Process. Appl.   118 (2008), no. 4, 614--628.
	
	
	
	\bibitem{NP05}
	D.~Nualart, G.~Peccati,
	{\it Central limit theorems for sequences
		of multiple stochastic integrals,}
	Ann. Probab.   33 (2005), no. 1, 177--193.
	
	\bibitem{NZ20a}
	D.~Nualart, G.~Zheng,
	{\it Averaging Gaussian functionals},
	Electron. J. Probab. 25 (2020), Paper No. 48, 54 pp.
	
	
	\bibitem{NZ20osc}
	D.~Nualart, G.~Zheng,
	{\it Oscillatory Breuer-Major theorem with application 
		to the random corrector problem},
	Asymptot. Anal. 119 (2020), no. 3-4, 281--300.
	
	
	\bibitem{NZ21}
	D.~Nualart, H.~Zhou,
	{\it Total variation estimates in the Breuer-Major theorem,}
	Ann. Inst. Henri Poincar\'e Probab. Stat. 57 (2021), no. 2, 740--777.
	
	
	
	\bibitem{PS24} 
	G.~Peccati, M.~Stecconi,
	{\it Nodal Volumes as Differentiable Functionals of Gaussian fields},
	arXiv:2403.20243 [math.PR].
	
	
	
	
	\bibitem{PT05}
	G.~Peccati, C.A.~Tudor,
	{\it  Gaussian limits for vector-valued multiple stochastic integrals,}
	Lecture Notes in Math., 1857
	Springer-Verlag, Berlin, 2005, 247--262.
	
	\bibitem{PT23}
	G. Peccati and N. Turchi, The discrepancy between min-max statistics of Gaussian and Gaussian-subordinated matrices, Stochastic Process. Appl. {\bf 158} (2023), 315--341; MR4536760
	
	
	%
	\bibitem{PV20}
	G.~Peccati, A.~Vidotto, 
	{\it Gaussian random measures generated by Berry's nodal sets,}
	J. Stat. Phys. 178 (2020), no.4, 996--1027.
	
	
	
	
	
	\bibitem{Ros81}
	
	M.~Rosenblatt,
	{\it   Limit theorems for Fourier transforms of functionals of 
		Gaussian sequences,} Z. Wahrsch. Verw. Gebiete (1981)  55, no.2, 
	123--132.
	
	
	
	
	\bibitem{Ros19}
	
	M.~Rossi,  
	{\it The defect of random hyperspherical harmonics},
	J. Theoret. Probab.   32 (2019), no. 4, 2135--2165.
	
	
	
	\bibitem{Smu25}
	
	K. Smutek, Fluctuations of the nodal number in the two-energy planar Berry's random wave model, ALEA Lat. Am. J. Probab. Math. Stat. {\bf 22} (2025), no.~1, 1--72; MR4860601
	
	
	\bibitem{Sze39}
	G. Szego. Orthogonal Polynomials. American Math. Soc: Colloquium publ. American Mathematical Society, 1939.
	
	\bibitem{Taq75}
	M.S.~Taqqu, 
	{\it Weak convergence to fractional Brownian motion and to the
		Rosenblatt process}, Zeitschrift f{\"u}r Wahrscheinlichkeitstheorie und
	Verwandte Gebiete  31 (1975), no.~4, 287--302.
	
	
	
	\bibitem{Taq77}
	M.S.~Taqqu, {\it Law of the iterated logarithm 
		for sums of non-linear functions of Gaussian
		variables that exhibit a long range dependence,}
	Z. Wahrsch. Verw. Gebiete (1977) 40, no.3,
	203--238.
	
	\bibitem{Taq79}
	
	M.S.~Taqqu,
	{\it  Convergence of integrated processes of 
		arbitrary Hermite rank,}
	Z. Wahrsch. Verw. Gebiete (1979) 50, no.1,   53--83.
	
	
	
	
	
	
	\bibitem{Tod19}
	A.P.~Todino, 
	{\it A Quantitative Central Limit Theorem for the Excursion Area of
		Random Spherical Harmonics over Subdomains of $\mathbb S^2$,}
	J. Math. Phys. 60 (2019), no. 2, 023505, 33 pp.
	
	
	
	
	
	
	\bibitem{Vid21}
	%
	A.~Vidotto, 
	{\it A note on the reduction principle for the nodal length 
		of planar random waves,}
	Statist. Probab. Lett. 174 (2021), Paper No. 109090, 5 pp.
	\bibitem{Wig10}
	%
	I.~Wigman, 
	{\it Fluctuations of the nodal length of random spherical harmonics},
	Comm. Math. Phys. 298 (2010), no.3, 787--831.
	%
	%
	\bibitem{ZX17}
	Y. Zhou and Y.~M. Xiao, Tail asymptotics for the extremes of bivariate Gaussian random fields, Bernoulli {\bf 23} (2017), no.~3, 1566--1598; MR3624871
	
	
	
	
	
	
	
	
	
	
	
\end{thebibliography}
\end{document}